\documentclass[11pt]{amsart}

\usepackage{verbatim,natbib}
\usepackage[active]{srcltx}
\usepackage[all]{xy}
\usepackage[a4paper,hcentering,vcentering]{geometry}
\newtheorem{theorem}{Theorem}
\newtheorem{corollary}[theorem]{Corollary}
\newtheorem{lemma}[theorem]{Lemma}
\newtheorem{proposition}[theorem]{Proposition}

\newtheoremstyle{simple}{2pt}{}{}{}{\bf}{}{5pt}{}
\theoremstyle{simple}

\newtheorem{example}[theorem]{Example}
\newtheorem{exercise}[theorem]{Exercise}

\newcommand{\nm}[1]{\left|{#1}\right|}

\newcommand{\C}{\ensuremath{\mathbb{C}}}
\newcommand{\D}{\ensuremath{\mathbb{D}}}
\newcommand{\N}{\ensuremath{\mathbb{N}}}
\newcommand{\R}{\ensuremath{\mathbb{R}}}
\newcommand{\T}{\ensuremath{\mathbb{T}}}
\newcommand{\Z}{\ensuremath{\mathbb{Z}}}

\newcommand{\Dr}{\ensuremath{\mathrm{D}}}
\newcommand{\Tr}{\ensuremath{\mathrm{T}}}

\renewcommand{\Re}{\ensuremath{\mathrm{Re}\,}}
\renewcommand{\Im}{\ensuremath{\mathrm{Im}\,}}

\newcommand{\gk}{\ensuremath{\mathfrak{g}}}

\DeclareMathOperator{\id}{id}

\begin{document}

\title{A simple proof of the invariant torus theorem}
\author{Jacques F{\'e}joz}
\date{May 2010}




\begin{abstract}
  We give a simple proof of Kolmogorov's theorem on the persistence of
  a quasiperiodic invariant torus in Hamiltonian systems. The theorem
  is first reduced to a well-posed inversion problem (Herman's normal
  form) by switching the frequency obstruction from one side of the
  conjugacy to another. Then the proof consists in applying a simple,
  well suited, inverse function theorem in the analytic category,
  which itself relies on the Newton algorithm and on interpolation
  inequalities. A comparison with other proofs is included in
  appendix.
\end{abstract}

\maketitle

\tableofcontents


\section{The invariant torus theorem}

Let $\mathcal{H}$ be the space of germs along $\Tr_0^n := \T^n \times
\{0\}$ of real analytic Hamiltonians in $\T^n \times \R^n =
\{(\theta,r)\}$ ($\T^n=\R^n/2\pi\Z^n$). The vector field associated
with $H \in \mathcal{H}$ is
$$\vec H : \quad \dot\theta = \partial_r H, \quad \dot
r=-\partial_\theta H.$$ 

For $\alpha\in\R^n$, let $\mathcal{K}$ be the affine subspace of
Hamiltonians $K\in \mathcal{H}$ such that $K|_{\Tr_0^n}$ is constant
(i.e.  $\Tr_0^n$ is invariant) and $\vec K|_{\Tr_0^n}=\alpha$. Those
Hamiltonians are characterized by their first order expansion along
$\Tr_0^n$, of the form $c + \alpha \cdot r$ for some $c\in\R$, that
is, their expansion is constant with respect to $\theta$ and the
coefficient of $r$ is $\alpha$. 

Let
$$\mathrm{D}_{\gamma,\tau} = \{\alpha \in \R^n, \; \forall k \in
\Z^n \setminus \{0\} \; |k \cdot \alpha| \geq \gamma |k|^{-\tau}\},
\quad |k|:= |k|_1 = |k_1| + \cdots + |k_n|.$$ If $\tau>n-1$, the set
$\cup_{\gamma >0} \mathrm{D}_{\gamma,\tau}$ has full measure
(\cite[p.~83]{Arn63}). See appendix~\ref{app:arithmetics}.

\begin{theorem}[\cite{Kol54,Chi08}]\label{thm:kolmogorov}
  Let $\alpha\in \mathrm{D}_{\gamma,\tau}$ and
  $K^o\in\mathcal{K}$ such that the averaged Hessian
  $$\int_{\T^n} \frac{\partial^2 K^o}{\partial r^2}(\theta,0)\, d\theta$$
  is non degenerate. Every $H\in\mathcal{H}$ close to $K^o$ possesses
  an $\alpha$-quasiperiodic invariant torus.
\end{theorem}

This theorem has far reaching consequences. In particular it has led
to a partial answer to the long standing question of the stability of
the Solar system (\cite{Arn64, Fej04, CC07}). See
\cite{Bos86,Sev03,Lla99} for references and background.

Kolmogorov's theorem is a consequence of the following normal form.
Let $\mathcal{G}$ be the space of germs along $\Tr_0^n$ of real
analytic exact symplectomorphisms $G$ in $\T^n \times \R^n$ of the
following form:
$$G(\theta,r) = (\varphi(\theta), {}^t\varphi'(\theta)^{-1}
(r+\rho(\theta))),$$ where $\varphi$ is a real analytic isomorphism of
$\T^n$ fixing the origin, and $\rho$ is an exact $1$-form on $\T^n$.

\begin{theorem}[Herman]\label{thm:herman0}
  Let $\alpha\in \mathrm{D}_{\gamma,\tau}$ and $K^o\in\mathcal{K}$.
  For every $H\in\mathcal{H}$ close enough to $K^o$, there exists a
  unique $(K,G,\beta) \in \mathcal{K} \times \mathcal{G} \times \R^n$
  close to $(K^o,\id,0)$ such that
  $$H=K \circ G + \beta \cdot r$$
  in some neighborhood of $G^{-1}(\T_0^n)$. Moreover, $\beta$ depends
  $C^1$-smoothly on $H$.
\end{theorem}

In other words, the orbits of Hamiltonians $K \in \mathcal{K}$ under
the action of symplectomorphisms of $\mathcal{G}$ locally form a
subspace of finite codimension $n$. The offset $\beta \cdot r$ usually
breaks the dynamical conjugacy between $K$ and $H$; hence Herman's
normal form is of geometrical nature and can be called a \emph{twisted
  conjugacy}. The strategy for deducing the existence of an
$H$-invariant torus (namely, $G^{-1}(\Tr^n_0)$) from that of a
$K$-invariant torus (namely, $\Tr^n_0$) is to show that $\beta$
vanishes on some subset of large measure in some parameter space (in
some cases, the frequency $\alpha$ cannot be fixed and needs to be
varied).

In the paper, $O(r^n)$ will denote the ideal of functions of
$(\theta,r)$ of the $n$-th order with respect to $r$.

\begin{proof}[Proof of theorem~\ref{thm:kolmogorov} assuming
  theorem~\ref{thm:herman0}]
  Let $K^o_2(\theta) := \frac{1}{2} \frac{\partial^2 K^o}{\partial
    r^2}(\theta,0)$. Let $F$ be the analytic function taking values
  among sym\-metric bilinear forms, which solves the cohomological
  equation $\mathcal{L}_\alpha F = K^o_2 - \int_{\T^n_0} K^o_2 \,
  d\theta$ (see lemma~\ref{lm:cohomological}), and $\varphi$ be the
  germ along $\Tr^n_0$ of the (well defined) time-one map of the flow
  of the Hamiltonian $F(\theta) \cdot r^2$. The map $\varphi$ is
  symplectic and restricts to the identity on $\Tr_0^n$. At the
  expense of substituting $K^o \circ \varphi$ and $H \circ \varphi$
  for $K^o$ and $H$ respectively, one can thus assume that
  $$K^o = c + \alpha \cdot r + Q \cdot r^2 +
  O(r^3), \quad Q:= \int_{\Tr_0^n} K^o_2(\theta) \, d\theta.$$ The
  germs so obtained from the initial $K^o$ and $H$ are close to one
  another.

  Consider the family of trivial perturbations obtained by translating
  $K^o$ in the direction of actions: $$K^o_R(\theta,r) :=
  K^o(\theta,R+r), \quad R\in\R^n,\; \mbox{$R$ small},$$ and its
  approximation obtained by truncating the first order jet of $K^o_R$
  along $\Tr^n_0$ from its terms $O(R^2)$ which possibly depend on
  $\theta$:
  $$\hat K^o_R (\theta,r) := \left(c + \alpha \cdot R \right) +
  \left(\alpha + 2Q \cdot R \right) \cdot r + O (r^2) = K^o_R +
  O(R^2).$$ For the Hamiltonian $\hat K^o_R$, $T^n_0$ is invariant and
  quasiperiodic of frequency $\alpha + 2Q \cdot R$. Hence the Herman
  normal form of $\hat K^o_R$ with respect to the frequency $\alpha$
  is
  $$\hat K^o_R = \left( \hat K^o_R-\hat\beta^o_R \cdot r \right) \circ
  \id + \hat \beta^o_R \cdot r, \quad \hat \beta^o_R:= 2 Q \cdot R.$$
  By assumption the matrix $\left.\frac{\partial \hat\beta^o}{\partial
      R}\right|_{R=0} = 2Q$ is invertible and the map $R \mapsto
  \hat\beta^o(R)$ is a local diffeomorphism.

  Now, theorem~\ref{thm:herman0} asserts the existence of an analogous
  map $R \rightarrow \beta(R)$ for $H_R$, which is a small
  $C^1$-perturbation of $R \mapsto \hat\beta^o(R)$, and thus a local
  diffeomorphism, with a domain having a lower bound locally uniform
  with respect to $H$. Hence if $H$ is close enough to $K^o$ there is
  a unique small $R$ such that $\beta=0$. For this $R$ the equality
  $H_R=K \circ G$ holds, hence the torus obtained by translating
  $G^{-1}(\Tr_0^n)$ by $R$ in the direction of actions is invariant
  and $\alpha$-quasiperiodic for $H$.
\end{proof}

\begin{exercise}
  Simplify this proof when $K^o = K^o(r)$ is integrable.
\end{exercise}

It is the aim of the rest of the paper to prove
theorem~\ref{thm:herman0}, by locally inverting some operator
$$\phi : (K,G,\beta) \mapsto H = K \circ G + \beta \cdot r$$
when $\alpha$ is diophantine. 


\section{Complexification and the functional setting}
\label{sec:setting}

\emph{For various sets $U$ and $V$, $\mathcal{A}(U,V)$ will denote the
  set of continuous maps $U \rightarrow V$ which are real analytic on
  the interior $\mathring{U}$, and
  $\mathcal{A}(U):=\mathcal{A}(U,\C)$.}

Recall notations for the abstract torus and its embedding in the phase
space:
$$\T^n = \R^n / 2\pi \Z^n \quad \mbox{and} \quad \Tr^n_0 = \T^n \times
\{0\} \subset \T^n \times \R^n.$$

Define complex extensions
$$\T^n_\C = \C^n / 2 \pi\Z^n \quad \mbox{and} \quad \Tr^n_\C = \T^n_\C
\times \C^n$$
as well as bases of neighborhoods
$$\T^n_s = \{\theta \in \T^n_\C, \; \max_{1\leq j\leq n} |\Im
\theta_j| \leq s\} \quad \mbox{and} \quad \Tr^n_s = \{ (\theta,r)\in
\Tr^n_\C, \, |(\theta,r)| \leq s\}, $$ 
with $|(\theta,r)| := \max_{1\leq j \leq n} \max \left( |\Im
  \theta_j|, |r_j| \right)$. 

\subsection{Spaces of Hamiltonians}

-- Let $\mathcal{H}_s = \mathcal{A}(\Tr^n_s)$, endowed with the Banach
norm
$$|H|_s := \sup_{(\theta,r)\in \Tr^n_s} |H(\theta,r)|,$$
so that $\mathcal{H}$ be the inductive limit of the spaces
$\mathcal{H}_s$. 

-- For $\alpha\in\R^n$, let $\mathcal{K}_s$ be the affine
subspace consisting of those $K \in \mathcal{H}_s$ such that
$K(\theta,r) = c + \alpha \cdot r + O(r^2)$ for some $c\in\R$.

-- If $G$ is a real analytic isomorphism on some open set of
$\Tr^n_\C$ and if $G$ is transverse to $\Tr_s^n$, let
$G^*\mathcal{A}(\Tr_s^n) := \mathcal{A}(G^{-1}(\Tr_s^n))$ be endowed
with the Banach norm
$$|H|_{G,s} := |H \circ G^{-1}|_s.$$

\subsection{Spaces of conjugacies}

\subsubsection{Diffeomorphisms of the torus}

Let $\mathcal{D}_s$ be the space of maps $\varphi\in
\mathcal{A}(\T^n_s,\T^n_\C)$ which are analytic isomorphisms from
$\mathring\T^n_s$ to their image and which fix the origin.

Let also
$$\chi_s :=\{v \in \mathcal{A}(\T^n_s)^n, \;
v(0)=0\}$$ be the space of vector fields on $\T^n_s$ which vanish at
$0$, endowed with the Banach norm
$$|v|_s := \max_{\theta\in\T^n_s} \max_{1\leq j \leq n}
|v_j(\theta)|.$$ 
According to corollary~\ref{prop:isom}, the map
$$\sigma \mathring B^\chi_{s+\sigma} := \{v \in \chi_{s+\sigma}, \;
|v|_s < \sigma\} \rightarrow \mathcal{D}_s, \quad v \mapsto \id + v$$
is defined and locally bijective. It endows $\mathcal{D}_s$ with a
local structure of Banach manifold in the neighborhood of the
identity.

We will consider the contragredient action of $\mathcal{D}_s$ on
$\Tr^n_s$ (with values in $\Tr^n_\C$)~:
$$\varphi (\theta,r) := (\varphi(\theta),{}^t \varphi'(\theta)^{-1}
\cdot r),$$ in order to linearize the dynamics on the alleged
invariant tori.

\subsubsection{Straightening tori}

Let $\mathcal{B}_s$ be the space of exact one-forms over $\T^n_s$,
with
$$|\rho|_s = \max_{\theta\in\T^n_s} \max_{1 \leq j \leq n}
|\rho_j(\theta)|, \quad \rho = (\rho_1,...,\rho_n).$$ We will consider
its action on $\Tr^n_s$ by translation of the actions:
$$\rho (\theta,r) := (\theta, r+\rho(\theta)),$$
in order to straighten the perturbed invariant tori. 

\subsubsection{Our space of conjugacies}

Let $\mathcal{G}_s = \mathcal{D}_s \times \mathcal{B}_s$,
identified with a space of Hamiltonian symplectomorphisms by
$$(\varphi,\rho)(\theta,r) := \varphi \circ \rho \, (\theta,r) =
(\varphi(\theta), {}^t\varphi'(\theta)^{-1}(r+\rho(\theta))).$$ Endow
its tangent space at the identity $T_{\id}\mathcal{G}_s = \gk_s :=
\chi_s \times \mathcal{B}_s$ with the norm
$$|\dot G|_s = |(v,\rho)|_s := \max( |v|_s, |\rho|_s),$$
and its tangent space at $G=(\varphi,\rho)$ with the norm
$$|\delta G|_s := |\delta G \circ G^{-1}|_s, \quad \delta G \in T_G
\mathcal{G}.$$ \emph{Here and elsewhere, the notation $\delta G$, as
  well as similar ones, should be taken as a whole; there is no
  separate $\delta \in \R$ in the present paper}.

Also consider the following neighborhoods of the identity:
$$\mathcal{G}_s^\sigma = \left\{G \in \mathcal{G}_s, \; \max_{(\theta,r)
    \in T^n_s} |(\Theta-\theta,R-r)| \leq \sigma, \;
  (\Theta,R)=G(\theta,r)\right\}, \quad \sigma >0.$$

The operators (commuting with inclusions of source and target spaces)
$$\phi_s : E_s:= \mathcal{K}_{s+\sigma} \times
\mathcal{G}_s^\sigma \times \R^n \rightarrow \mathcal{H}_s, \quad
(K,G,\beta) \mapsto K \circ G + \beta \cdot r$$ are now defined.

\section{Local twisted conjugacy of Hamiltonians}

\begin{theorem}\label{thm:herman}
  Let $\alpha\in \Dr_{\gamma,\tau}$. For all $0 < s < s+\sigma < 1$,
  $\phi_{s+\sigma}$ has a local inverse: if $|H-K^o|_{s+\sigma}$ is
  small, there is a unique $(K,G,\beta)\in E_{s}$, $|\cdot|_{s}$-close
  to $(K^o,\id,0)$ such that $H=K \circ G + \beta \cdot r$.  Moreover
  $\beta \circ \phi^{-1}$ is a $C^1$-function locally in the
  neighborhood of $K^o$ in $\mathcal{H}_{s+\sigma}$.
\end{theorem}

This entails theorem~\ref{thm:herman0} and itself follows from the
inverse function theorem of appendix~\ref{app:submersion}, from
lemma~\ref{lm:lip} (for the uniqueness) and from
corollary~\ref{cor:smoothness} (for the smoothness of $\beta \circ
\phi^{-1}$). We will now check the two main hypotheses of
appendix~\ref{app:submersion} (one on $\phi'^{-1}$ and one on
$\phi''$).

Let $\mathcal{L}_\alpha$ be the Lie derivative operator in the
direction of the constant vector field $\alpha$~:
$$\mathcal{L}_\alpha : \mathcal{A}(\T^n_s) \rightarrow
\mathcal{A}(\T^n_s), \quad f \mapsto f' \cdot \alpha = \sum_{1\leq j
  \leq n} \alpha_j \frac{\partial f}{\partial \theta_j}.$$ We will
need the following classical lemma in two instances in the proof of
lemma~\ref{lm:inverse}.

\begin{lemma}[Cohomological equation]\label{lm:cohomological}
  If $g\in \mathcal{A}(\T^n_{s+\sigma})$ has $0$-average ($\int_\T g
  \, d\theta=0$), there exists a unique function $f\in
  \mathcal{A}(\T^n_s)$ of $0$-average such that $\mathcal{L}_\alpha f
  = g$, and there exists a $C_0=C_0(n,\tau)$ such that, for any
  $\sigma$:
  $$|f|_s \leq C_0\gamma^{-1}\sigma^{-\tau-n} |g|_{s+\sigma}.$$
\end{lemma}

\begin{proof}
  Let $g(\theta) = \sum_{k\in \Z^n \setminus \{0\}} g_k \, e^{i k
    \cdot \theta}$ be the Fourier expansion of $g$.  The unique formal
  solution to the equation $\mathcal{L}_\alpha f = g$ is given by
  $f(\theta) = \sum_{k\in \Z^n \setminus \{0\}} \frac{g_k}{i \, k
    \cdot \alpha} \, e^{i\, k \cdot \theta}$.

  Since $g$ is analytic, its Fourier coefficients decay exponentially:
  we find  
  $$|g_k| = \left| \int_{\T^n} g(\theta) \, e^{-ik \cdot \theta} \,
    \frac{d\theta}{2\pi} \right| \leq |g|_{s+\sigma} e^{-|k|
    (s+\sigma)}$$ by shifting the torus of integration to a torus $\Im
  \theta_j =\pm (s+\sigma)$.

  Using this estimate and replacing the small denominators $k \cdot
  \alpha$ by the estimate defining the diophantine property of
  $\alpha$, we get
  \begin{eqnarray*}
    |f|_s 
    &\leq &\frac{|g|_{s+\sigma}}{\gamma} \sum_k  |k|^\tau \,
    e^{-|k|\, \sigma}\\
    &\leq &\frac{2^n |g|_{s+\sigma}}{\gamma} \sum_{\ell \geq
      1}
    \begin{pmatrix}
      \ell+n-1\\
      \ell
    \end{pmatrix}
    \ell^\tau \, e^{-\ell\, \sigma} \leq \frac{4^n
      |g|_{s+\sigma}}{\gamma \, (n-1)!} \sum_{\ell}
    (\ell+n-1)^{\tau+n-1}\, e^{-\ell \, \sigma},  
  \end{eqnarray*}
  where the latter sum is bounded by
  \begin{eqnarray*}
   \int_1^\infty (\ell+n-1)^{\tau+n-1}e^{-(\ell-1)\sigma} \, d\ell
   &=&\sigma^{-\tau-n} e^{n\sigma} \int_{n\sigma}^\infty
   \ell^{\tau+n-1} e^{-\ell} \, d\ell\\
   &<&\sigma^{-\tau-n} e^{n\sigma} \int_0^\infty \ell^{\tau+n-1}
   e^{-\ell} \, d\ell = \sigma^{-\tau-n} e^{n\sigma} \Gamma(\tau+n).
 \end{eqnarray*}
   Hence $f$ belongs to $\mathcal{A}(\T^n_s)$ and satisfies the wanted
  estimate. 
\end{proof}

We will write $x=(K,G,\beta,c)$, $\delta x = (\delta K, \delta G,
\delta\beta, \delta c)$ and $\delta \hat x = (\delta \hat K, \delta \hat
G, \delta\hat\beta, \delta\hat c)$.

Fix $0 < s < s+\sigma < 1$.

\begin{lemma} \label{lm:inverse} There exists $C'>0$ which is locally
  uniform with respect to $x \in E_s$ in the neighborhood of $G=\id$
  such that the linear map $\phi'(x)$ has an inverse $\phi'(x)^{-1}$
  satisfying
  $$\nm{\phi'(x)^{-1} \cdot \delta H}_{s} \leq 
  \sigma^{-\tau-n-1} C' \nm{\delta H}_{G,s+\sigma}.$$
\end{lemma}

\begin{proof}
  A function $\delta H \in G^*\mathcal{A}(T_{s+\sigma})$ being given,
  we want to solve the equation
  $$\delta\phi(x) \cdot \delta x = \delta K \circ G + K' \circ G \cdot
  \delta G + \delta\beta \cdot r + \delta c = \delta H,$$ for the
  unknowns $\delta K\in T_K \mathcal{K}_s \subset
  \mathcal{A}(\Tr^n_s)$, $\delta G \in T_G \mathcal{G}_s$,
  $\delta\beta \in \R^n$ and $\delta c \in \R$, or, equivalently,
  after composing with $G^{-1}$ to the right,
  $$\delta K + K' \cdot \dot G + \delta\beta \cdot r \circ G^{-1} +
  \delta c = \dot H,$$ where we have set $\dot G:=\delta G \circ
  G^{-1} \in \gk_s$ and $\dot H:= \delta H \circ G^{-1} \in
  \mathcal{A}(\Tr^n_s)$.

  More specifically, $G^{-1}$ and $\dot G$ are of the form 
  $$G^{-1}(\theta,r) = (\varphi^{-1}(\theta), {}^t\varphi' \circ
  \varphi^{-1}(\theta) \cdot r - \rho \circ \varphi^{-1}(\theta)),
  \quad \dot G = (\dot\varphi, \dot\rho - r \cdot \dot\varphi'),$$  
  where $\dot\varphi\in \chi_{s+\sigma}$ and
  $\dot\rho\in \mathcal{B}_{s+\sigma}$, 
  and we can expand
  $$K = \alpha \cdot r + K_2(\theta) \cdot r^2 + O(r^3) \quad
  \mbox{and} \quad \dot H = \dot H_0(\theta) + \dot H_1(\theta) \cdot
  r + O(r^2).$$
  The equation becomes
  \begin{multline} 
    \label{eq:cohomological}
    \left[\dot \rho \cdot \alpha + \delta c - \rho \circ \varphi^{-1}
      \cdot \delta \beta \right] + r \cdot \left[ -\dot\varphi' \cdot
      \alpha + \varphi' \circ \varphi^{-1} \cdot
      \delta \beta + 2 K_2 \cdot \dot \rho \right] + \\
    \dot K = \dot H + O(r^2),
  \end{multline}
  where the term $O(r^2)$ in the right hand side depends only on $K$
  and $\dot G$, and not on $\dot K$. The equation turns out to be
  triangular in the five unknowns. The existence and uniqueness of a
  solution with the wanted estimate follows from repeated applications
  of lemma~\ref{lm:cohomological} and Cauchy's inequality:

  -- The average over $\Tr_0^n$ of the first order terms with respect
  to $r$ in equation~(\ref{eq:cohomological}) yields
  $$\delta\beta = \left(\int_{\T^n} \varphi' \circ \varphi^{-1} \,
    d\theta \right)^{-1} \cdot \int_{\Tr_0^n} \dot H_1 \, d\theta,$$
  which does exist if $\varphi$ is close to the identity
  (proposition~\ref{prop:isom}).

  -- Similarly, the average of the restriction to $\Tr_0^n$
  of~(\ref{eq:cohomological}) yields:
  $$\delta c = \int_{\Tr^n_0} \dot H_0   \, d\theta + \int_{\T_0^n}
  \rho \circ \varphi^{-1} \, d\theta \cdot \delta\beta.$$ 
  
  -- Next, the restriction to $\Tr_0^n$ of~(\ref{eq:cohomological})
  can be solved uniquely with respect to $\delta\rho$ according to
  lemma~\ref{lm:cohomological} (applied with $\rho=f'$).

  -- The part of degree one can then be solved for $\dot\varphi$
  similarly.

  -- Terms of order $\geq 2$ in $r$ determine $\dot K$. 
\end{proof}

\begin{lemma}\label{lm:seconde}
  There exists a constant $C''>0$ which is locally uniform with
  respect to $x \in E_{s+\sigma}$ in the neighborhood of $G=\id$ such
  that the bilinear map $\phi''(x)$ satisfies
  $$\nm{\phi''(x) \cdot \delta x \otimes \delta \hat x}_{G,s}\leq
  \sigma^{-1} C'' \, |\delta x|_{s+\sigma} |\delta\hat
  x|_{s+\sigma}.$$
\end{lemma}

\begin{proof}
  Differentiating $\phi$ twice yields
  $$\phi''(x) \cdot \delta x \otimes \delta \hat x = \delta K' \circ
  G \cdot \delta G + \delta\hat K' \circ G \cdot \delta G + K'' \circ
  G \cdot \delta G \otimes \delta \hat G,$$ whence the estimate.
\end{proof}

\setcounter{section}{0}
\renewcommand{\thesection}{\Alph{section}}
\addcontentsline{toc}{section}{Appendices}
\section{An inverse function theorem}
\label{app:submersion}

Let $E=(E_s)_{0<s<1}$ be a decreasing family of Banach spaces with
increasing norms $| \cdot |_s$, and $\epsilon B^E_s = \{x\in E_s,
\nm{x}_s< \epsilon\}$, $\epsilon>0$, be its balls centered at $0$. 

Let $(F_s)$ be an analogous family. Endow $F$ with additional norms
$\nm{\cdot}_{x,s}$, $x\in E_s$, $0<s<1$, satisfying
$$|y|_{0,s} = |y|_s \quad \mbox{and} \quad  \nm{y}_{x',s} \leq
\nm{y}_{x,s+\nm{x'-x}_s}.$$  
These norms allow for dealing with composition operators without
artificially loosing some fixed ``width of analyticity'' $\sigma$ at
each step of the Newton algorithm.

Let $\phi : \sigma B^E_{s+\sigma} \rightarrow F_s$, $s<s+\sigma$,
$\phi(0)=0$, be maps commuting with inclusions, twice differentiable,
such that the differential $\phi'(x) : E_{s+\sigma} \rightarrow F_s$
has a right inverse $\phi'(x)^{-1} : F_{s+\sigma} \rightarrow E_s$,
and
$$\left\{
  \begin{array}[c]{ll}
    |\phi'(x)^{-1}\eta|_{s} \leq C' \sigma^{-\tau'}
    |\eta|_{x,s+\sigma}\\ 
    |\phi''(x) \xi^{\otimes 2}|_{x,s} \leq
    C''\sigma^{-\tau''} |\xi|_{s+\sigma}^2 \quad (\forall s, \sigma,
    x, \xi, \eta)
  \end{array}\right. $$
with $C',C'',\tau',\tau'' \geq 1$. Let $C:=C'C''$ and
$\tau:=\tau'+\tau''$. 

\begin{theorem} \label{thm:inversion} $\phi$ is locally surjective
  and, more precisely, for any $s$, $\eta$ and $\sigma$ with $\eta <
  s$,
  $$\epsilon B^F_{s+\sigma} \subset \phi \left( \eta B^E_s \right),
  \quad \epsilon := 2^{-8\tau} C^{-2}\sigma^{2\tau} \eta.$$
\end{theorem}

In other words, $\phi$ has a right-inverse $\psi : \epsilon
B^F_{s+\sigma} \rightarrow \eta B^E_{s}$. 

\begin{proof}
  Some numbers $s$, $\eta$ and $\sigma$ and $y \in B^F_{s+\eta}$ being
  given, let
  $$f : \sigma B^E_{s+\eta+\sigma} \rightarrow E_s, \quad x
  \mapsto x + \phi'(x)^{-1} (y-\phi(x))$$ 
  and 
  $$Q : \sigma B^E_{s+\sigma} \times \sigma B^E_{s+\sigma}
  \rightarrow F_s, \quad (x,\hat x) \mapsto  \phi(\hat
  x)-\phi(x)-\phi'(x)(\hat x-x).$$  

  \begin{lemma} \label{lm:Q} The function $Q$ satisfies: $\nm{Q(x,\hat
      x)}_{x,s} \leq 2^{-1}C''\sigma^{-\tau''} \nm{\hat
      x-x}_{s+\sigma+|\hat x-x|_s}^2$.
  \end{lemma}

  \begin{proof}[Proof of the lemma]
    Let  $\hat x_t:=(1-t)x + t\hat x$. Taylor's formula yields
    $$Q(x,\hat x) = \int_0^1 (1-t) \, \phi''(\hat x_t) \, (\hat
    x-x)^2\, dt,$$ 
    hence
    $$\nm{Q(x,\hat x)}_{x,s} \leq \int_0^1 (1-t) \left|\phi''(\hat x_t)
      (\hat x-x)^2\right|_{x,s}\, dt \leq \int_0^1 (1-t)
    \left|\phi''(\hat x_t) (\hat x-x)^2\right|_{\hat x_t,s+|\hat
      x_t-x|_s}\, dt,$$ whence the estimate.
  \end{proof}

  Now, let $s$, $\eta$ and $\sigma$ be fixed, with $\eta < s$ and
  $y\in \epsilon B^F_{s+\sigma}$ for some $\epsilon$. We will see that
  if $\epsilon$ is small enough, the sequence $x_0=0$, $x_n:=f^n(0)$
  is defined for all $n\geq 0$ and converges towards some preimage $x
  \in \eta B^E_s$ of $y$ by $\phi$.

  Let $(\sigma_n)_{n\geq 0}$ be a sequence of positive real numbers
  such that $3 \sum \sigma_n = \sigma$, and $(s_n)_{n\geq 0}$ be the
  sequence decreasing from $s_0:= s+\sigma$ to $s$ defined by
  induction by the formula $s_{n+1} = s_n - 3 \sigma_n$.

  Assuming the existence of $x_0, ..., x_{n+1}$, we see that
  $\phi(x_k) = y + Q(x_{k-1},x_k)$, hence
  $$x_{k+1}-x_k = \phi'(x_k)^{-1}(y-\phi(x_k)) =
  -\phi'(x_k)^{-1}Q(x_{k-1},x_k) \qquad (1\leq k \leq n).$$ Further
  assuming that $|x_{k+1}-x_k|_{s_k} \leq \sigma_k$, the estimate of the
  right inverse and lemma~\ref{lm:Q} entail that
  $$|x_{n+1}-x_n|_{s_{n+1}} \leq c_n |x_n-x_{n-1}|^2_{s_n} \leq
  \cdots \leq c_n c_{n-1}^2 \cdots c_1^{2^{n-1}} \,
  |x_1|_{s_1}^{2^{n-1}}, \quad c_k:=2^{-1}C\sigma_k^{-\tau}.$$ The
  estimate
  $$|x_1|_{s_1} \leq C' (3\sigma_0)^{-\tau'} |y|_{s_0} \leq
  2^{-1}C\sigma_0^{-\tau} \epsilon = c_0\epsilon$$
  and the fact, to be checked later, that $c_k\geq 1$ for all $k\geq
  0$, show~:
  $$|x_{n+1}-x_n|_{s_{n+1}} \leq \left( \epsilon \prod_{k\geq 0}
    c_k^{2^{-k}} \right)^{2^n}.$$ Since $\sum_{n\geq 0} \rho^{2^n}
  \leq 2\rho$ if $2\rho\leq 1$, and using the definition of constants
  $c_k$'s, we get a sufficient condition to have all $x_n$'s defined and
  to have $\sum |x_{n+1}-x_n|_s \leq \eta$:
  \begin{equation}
    \label{eq:epsilon}
    \epsilon = \frac{\eta}{2} \prod_{k\geq 0} c_k^{-2^{-k}} =
    \frac{2\eta}{C^2} \prod_{k\geq 0} \sigma_k^{\tau 2^{-k}}.
  \end{equation}

  Maximizing the upper bound of $\epsilon$ under the constraint
  $3\sum_{n\geq 0} \sigma_n = \sigma$ yields $\sigma_k :=
  \frac{\sigma}{6}2^{-k}$. A posteriori it is straightforward that
  $|x_{n+1}-x_n|_{s_{n}} \leq \sigma_n$ (as earlier assumed to apply
  lemma~\ref{lm:Q}) and $c_n\geq 1$ for all $n\geq 0$. Besides, using
  that $\sum k 2^{-k} = \sum 2^{-k} = 2$ we get
  $$\frac{\eta}{2} \prod_{k\geq 0} c_k^{-2^{-k}} = \frac{\eta}{2}
  \prod_{k\geq 0} \frac{1}{2^{\tau k2^{-k}}} \left( \frac{2}{C} \left(
      \frac{\sigma}{6} \right)^\tau \right)^{2^{-k}} =
  \frac{2\eta}{C^2} \left( \frac{\sigma}{12} \right)^{2\tau} >
  \frac{\sigma^{2\tau} \eta}{2^{8\tau}C^2},$$
  whence the theorem.
\end{proof}

\begin{exercise}
  The domain of $\psi$ contains $\epsilon B^F_S$, $\epsilon =
  2^{-12\tau}\tau^{-1}C^{-2}S^{3\tau}$, for any $S$.
\end{exercise}

\begin{proof}
  The above function $\epsilon(\eta,\sigma) = 2^{-8\tau} C^{-2}
  \sigma^{2\tau} \eta$ attains is maximum with respect
  to $\eta<s$ for $\eta=s$. Besides, under the constraint $s+\sigma =
  S$ the function $\epsilon(s,\sigma)$ attains its maximum when
  $\sigma = 2\tau s$ and $s=\frac{S}{1+2\tau}$. Hence, $S$ being
  fixed, the domain of $\psi$ contains $\epsilon B^F_S$ if
  $$\epsilon < 2^{-8\tau} C^{-2} \frac{S}{1+2\tau} \left( \frac{2\tau
      S}{12(1+2\tau)} \right)^{2\tau}.$$ Given that $S < 1 < \tau$ by
  hypothesis, it suffices that $\epsilon$ be equal to the stated
  value.
\end{proof}

\subsection{Regularity of the right-inverse}
\label{subsec:regularity}

In the proof of theorem~\ref{thm:inversion} we have built right
inverses $\psi : \epsilon B^F_{s+\eta+\sigma} \rightarrow \eta
B^E_{s+\eta}$, of $\phi$, commuting with inclusions. The estimate
given in the statement shows that $\psi$ is continuous at $0$; due to
the invariance of the hypotheses of the theorem by small translations,
$\psi$ is locally continuous.

We further make the following two asumptions: 

-- The maps $\phi'(x)^{-1} : F_{s+\sigma} \rightarrow E_s$ are left
(as well as right) inverses (in theorem~\ref{thm:herman} we have
restricted to an adequate class of symplectomorphisms);

-- The scale $(| \cdot |_s)$ of norms of $(E_s)$ satisfies some
interpolation inequality:
$$|x|_{s+\sigma}^2 \leq |x|_s \, |x|_{s+\tilde\sigma} \quad
\mbox{for all $s$, $\sigma$, }
\tilde\sigma=\sigma\left(1+\frac{1}{s}\right)$$ (according to the
remark after corollary~\ref{cor:interpolation}, this estimate is
satisfied in the case of interest to us, since $\sigma + \log
(1+\sigma/s) \leq \tilde\sigma$). 

\begin{lemma}[Lipschitz regularity] \label{lm:lip}
  If $\sigma < s$ and $y, \hat y \in \epsilon B^F_{s+\sigma}$ with
  $\epsilon = 2^{-14\tau} C^{-3} \sigma^{3\tau}$,
  $$|\psi(\hat y) - \psi(y)|_s \leq C_L |\hat  y - y|_{s+\sigma},
  \quad C_L = 2C'\sigma^{-\tau'}.$$ In particular, $\psi$ is the
  unique local right inverse of $\phi$, and hence is also its local
  left inverse.
\end{lemma}

\begin{proof}
  Fix $\eta < \zeta < \sigma < s$; the impatient reader can readily
  look at the end of the proof how to choose the auxiliary parameters
  $\eta$ and $\zeta$ more precisely.

  Let $\epsilon = 2^{-8\tau} C^{-2} \zeta^{2\tau}\eta$, and $y,
  \hat y \in \epsilon B^F_{s+\sigma}$. According to
  theorem~\ref{thm:inversion}, $x:=\psi(y)$ and $\hat x:=\psi(\hat y)$
  are in $\eta B^E_{s+\sigma-\zeta}$, provided the condition, to be
  checked later, that $\eta < s+\sigma-\zeta$. In particuliar, we
  will use a priori that
  $$|\hat x - x|_{s+\sigma-\zeta} \leq |\hat x|_{s+\sigma-\zeta} +
  |x|_{s+\sigma-\zeta} \leq 2 \eta.$$

  We have
  \begin{eqnarray*}
    \hat x -x
    &=&\phi'(x)^{-1}\phi'(x)(\hat x -x)\\
    &=&\phi'(x)^{-1}\left( \hat y - y - Q(x,\hat x) \right)
  \end{eqnarray*}
  and, according to the assumed estimate on $\phi'(x)^{-1}$ and to
  lemma~\ref{lm:Q},
  \begin{eqnarray*}
    |\hat x -x|_s
    &\leq&C'\sigma^{-\tau'}|\hat y-y|_{s+\sigma} +
    2^{-1}C\zeta^{-\tau} |\hat x - x|_{s+2\eta+|\hat x-x|_s}^2.
  \end{eqnarray*}
  In the norm index of the last term, we will coarsely bound $|\hat
  x-x|_s$ by $2\eta$. Additionally using the interpolation inequality:
  $$|\hat x - x|_{s+4\eta}^2\leq |\hat x - x|_s |\hat x -
  x|_{s+\tilde\sigma}, \quad \tilde\sigma = 4\eta \left( 1+\frac{1}{s}
  \right),$$ 
  yields
  $$\left( 1 - 2^{-1}C \zeta^{-\tau} |\hat x - x|_{s+\tilde\sigma}
  \right) |\hat x - x|_s \leq C' \sigma^{-\tau'} |\hat y -
  y|_{s+\sigma}.$$ 
  Now, we want to choose $\eta$ small enough so that 
  
  -- first, $\tilde\sigma \leq \sigma - \zeta$, which implies $|\hat x
  - x|_{s+\tilde\sigma} \leq 2\eta$. By definition of $\tilde\sigma$,
  it suffices to have $\eta \leq \frac{\sigma-\zeta}{4(1+1/s)}$.

  -- second, $2^{-1}C\zeta^{-\tau} \, 2\eta \leq 1/2$, or $\eta \leq
  \frac{\zeta^\tau}{2C}$, which implies that $2^{-1}C\zeta^{-\tau}
  |\hat x - x|_{s+\tilde\sigma} \leq 1/2$, and hence $|\hat x-x|_{s}
  \leq 2 C'\sigma^{-\tau'}|\hat y-y|_{s+\sigma}$.

  A choice is $\zeta = \frac{\sigma}{2}$ and $\eta =
  \frac{\sigma^\tau}{16C} < s$, whence the value of $\epsilon$ in the
  statement.
\end{proof}

\begin{proposition}[Smoothness] \label{prop:regularity} For every
  $\sigma<s$, there exists $\epsilon,C_1$ such that for every $y, \hat
  y \in \epsilon B^F_{s+\sigma}$,
  $$|\psi(\hat y) - \psi(y) - \phi'(\psi(y))^{-1}(\hat y - y)|_s \leq C_1
  |\hat y - y|_{s+\sigma}^2.$$ Moreover, the map $\psi' : \epsilon
  B^F_{s+\sigma} \rightarrow L(F_{s+\sigma}, E_{s})$ defined locally by
  $\psi'(y) = \phi'(\psi(y))^{-1}$ is continuous.
\end{proposition}

\begin{proof}
  Fix $\epsilon$ as in the previous proof and $y, \hat y\in
  \varepsilon B^F_{s+\sigma}$. Let $x = \psi(y)$, $\eta=\hat y-y$,
  $\xi = \psi(y+\eta) - \psi(y)$ (thus $\eta = \phi(x+\xi)-\phi(x)$),
  and $ \Delta :=\psi(y+\eta) - \psi(y) - \phi'(x)^{-1}\eta$.
  Definitions yield
  \begin{eqnarray*}
    \Delta&=&\phi'(x)^{-1} \left( \phi'(x)\xi - \eta \right)
    =- \phi'(x)^{-1} Q(x,x+\xi).
  \end{eqnarray*}
  Using the estimates on $\phi'(x)^{-1}$ and $Q$ and the latter lemma, 
  $$|\Delta|_s \leq C_1 |\eta|_{s+\sigma'}^2$$  
  for some $\sigma'$ tending to $0$ when $\sigma$ itself tends to $0$,
  and for some $C_1>0$ depending on $\sigma$. Up the substitution of
  $\sigma$ by $\sigma'$, the estimate is proved.

  The inversion of linear operators between Banach spaces being
  analytic, $y \mapsto \phi(\psi(y))^{-1}$ is continuous in the stated
  sense. 
\end{proof}

\begin{corollary}~\label{cor:smoothness}
  If $\pi \in L(E_s,V)$ is a family of linear maps, commuting with
  inclusions, into a fixed Banach space $V$, then $\pi \circ \psi$ is
  $C^1$ and $(\pi \circ \psi)' = \pi \cdot \phi' \circ \psi$.
\end{corollary}

This corollary is used with $\pi : (K,G,\beta) \mapsto \beta$ in the
proof of theorem~\ref{thm:herman}.

\section{Some estimates on analytic isomorphisms}
\label{app:inversion}

In this appendix, we give a quantitative inverse function theorem for
real analytic isomorphisms on $\T^n_s$. This is used in
section~\ref{sec:setting}, to parametrize locally $\mathcal{D}_s$ by
vector fields, and, in lemma~\ref{lm:cohomological}, to solve the
cohomological equation for the frequency offset $\delta\beta$.

Recall that we have set $\T^n_s := \{ \theta \in \C^n/2\pi\Z^n, \quad
\max_{1\leq j \leq n} |\Im\theta_j|\leq s\}$. We will denote by $p
: \R^n_s := \R^n \times i[-s,s]^n \rightarrow \T^n_s$ its universal
covering.

\begin{proposition}\label{prop:isom}
  Let $v\in \mathcal{A}(\T^n_{s+2\sigma},\C^n)$, $|v|_{s+2\sigma} <
  \sigma$. The map $\id + v : \T^n_{s+2\sigma} \rightarrow
  \R^n_{s+3\sigma}$ induces a map $\varphi : \T^n_{s+2\sigma}
  \rightarrow \T^n_{s+3\sigma}$ whose restriction $\varphi :
  \T^n_{s+\sigma} \rightarrow \T^n_{s+2\sigma}$ has a unique right
  inverse $\psi : \T^n_{s} \rightarrow \T^n_{s+\sigma}$:
  $$\xymatrix{
    \T^n_{s+\sigma} \ar^\varphi @{^{(}->} [r] &\T^n_{s+2\sigma}\\
    &\ar^\psi @{^{(}->} [ul] \T_s^n \ar @{^{(}->} [u] }.$$ 
  Furthermore, 
  $$|\psi-\id|_s \leq |v|_{s+\sigma}$$
  and, provided $2 \sigma^{-1} |v|_{s+2\sigma} \leq 1$,
  $$|\psi'-\id| \leq 2 \sigma^{-1} |v|_{s+2\sigma}.$$  
\end{proposition}

\begin{proof}
  Let $\Phi : \R^n_{s+2\sigma} \rightarrow \R^n_{s+3\sigma}$ be a
  continuous lift of $\id+v$ and $k \in M_n(\Z)$, $k(l) := \Phi(x+l) -
  \Phi(x)$.
  \begin{enumerate}
  \item \textit{Injectivity of $\Phi : \R^n_{s+\sigma} \rightarrow
      \R^n_{s+2\sigma}$. }  Suppose that $x, \hat x\in
    \R^n_{s+\sigma}$ and $\Phi(x) = \Phi(\hat x)$.  By the mean value
    theorem,
    $$|x-\hat x| = |v(p\hat x) -v(p x)| \leq |v'|_{s+\sigma}
    |x-\hat x|,$$ and, by Cauchy's inequality,
    $$|x-\hat x| \leq \frac{|v|_{s+2\sigma}}{\sigma} |x-\hat x| <
    |\hat x - x|,$$ hence $x=\hat x$.

  \item \textit{Surjectivity of $\Phi$: $\R^n_s \subset
      \Phi(\R^n_{s+\sigma})$. }  For any given $y\in \R^n_s$, the
    contraction
    $$f : \R^n_{s+\sigma} \rightarrow \R^n_{s+\sigma}, \quad x
    \mapsto y-v(x)$$ has a unique fixed point, which is a pre-image of
    $y$ by $\Phi$.

  \item \textit{Injectivity of $\varphi : \T^n_{s+\sigma} \rightarrow
      \T^n_{s+2\sigma}$. }  Suppose that $p x$, $p\hat x \in
    \R^n_{s+\sigma}$ and $\varphi(p x) = \varphi(p \hat x)$, i.e.
    $\Phi(x) = \Phi(\hat x) + \kappa$ for some $\kappa \in \Z^n$. That
    $k$ be in $GL(n,\Z)$, follows from the invertibility of $\Phi$.
    Hence, $\Phi \left(x-k^{-1}(\kappa) \right) = \Phi(\hat x)$, and,
    due to the injectivity of $\Phi$, $p x = p \hat x$.

  \item \textit{Surjectivity of $\varphi : \T^n_s \subset \varphi
      (\T^n_{s+\sigma})$. } This is a trivial consequence of that of
    $\Phi$.

  \item \textit{Estimate on $\psi:= \varphi^{-1} : \T_s^n \rightarrow
      \T_{s+\sigma}^n$. } Note that the wanted estimate on $\psi$ is
    in the sense of $\Psi := \Phi^{-1} : \R_s^n \rightarrow
    \R_{s+\sigma}^n$. If $y\in \R^n_s$,
    $$\Psi(y) - y = - v(p\Psi(y)),$$
    hence $|\Psi- \id|_{s} \leq |v|_{s+\sigma}$.

  \item \textit{Estimate on $\psi'$. } We have $\psi' = \varphi'^{-1}
    \circ \varphi$, where $\varphi'^{-1}(x)$ stands for the inverse of
    the map $\xi \mapsto \varphi'(x) \cdot \xi$. Hence
    $$\psi' - \id = \varphi'^{-1} \circ \varphi - \id,$$
    and, under the assumption that $2\sigma^{-1} |v|_{s+2\sigma} \leq
    1$,
    $$|\psi' - \id|_{s} \leq |\varphi'^{-1}-\id|_{s+\sigma} \leq
    \frac{|v'|_{s+\sigma}}{1-|v'|_{s+\sigma}} \leq
    \frac{\sigma^{-1}|v|_{s+2\sigma}}{1-\sigma^{-1}|v|_{s+2\sigma}}
    \leq 2 \sigma^{-1} |v|_{s+2\sigma}.$$
  \end{enumerate}
\end{proof}

\section{Interpolation of spaces of analytic functions}
\label{app:convexity}

In this section we prove some Hadamard interpolation inequalities,
which are used in~\ref{subsec:regularity}.

Recall that we denote by $\T_\C^n$ the infinite annulus $\C^n
/2\pi\Z^n$, by $\T_s^n$, $s>0$, the bounded sub-annulus $\{ \theta \in
\T_\C^n, \; |\Im \theta_j| \leq s, \; j=1...n\}$ and by $\D_t^n$,
$t>0$, the polydisc $\{ r\in \C^n, \; |r_j| \leq t, \; j=1...n\}$. The
supremum norm of a function $f \in \mathcal{A}(\T^n_s \times \D^n_t)$
will be denoted by $|f|_{s,t}$.

Let $0 < s_0 \leq s_1$ and $0 < t_0 \leq t_1$ be such that
$$\log \frac{t_1}{t_0} =  s_1-s_0.$$
Let also $0 \leq \rho \leq 1$ and
$$s = (1-\rho) s_0 + \rho s_1 \quad \mbox{and} \quad t =
t_0^{1-\rho} t_1^\rho.$$

\begin{proposition} \label{prop:domains} If $f \in
  \mathcal{A}(\T_{s_1}^n \times \D_{t_1}^n)$,
  $$|f|_{s,t} \leq |f|_{s_0,t_0}^{1-\rho} \, |f|_{s_1,t_1}^{\rho}.$$ 
\end{proposition}

\begin{proof}
  Let $\tilde f$ be the function on $\T^n_{s_1} \times \D^n_{t_1}$,
  constant on $2n$-tori of equations $(\Im \theta,r) =cst$, defined by
  $$\tilde f(\theta,r) = \max_{\mu,\nu\in \T^n} \left| f\left((\pm
      \theta_1 + \mu_1,..., \pm \theta_n + \mu_n), \left(r_1 \,
      e^{i\nu_1}, ..., r_n \, e^{i\nu_n}\right)\right)\right|$$ (with
  all possible combinations of signs). Since $\log |f|$ is subharmonic
  and $\T^{2n}$ is compact, $\log \tilde f$ too is upper
  semi-continuous. Besides, $\log \tilde f$ satisfies the mean
  inequality, hence is plurisubharmonic.

  By the maximum principle, the restriction of $|f|$ to $\T^n_s \times
  \D^n_t$ attains its maximum on the distinguished boundary of $\T^n_s
  \times \D^n_t$. Due to the symmetry of $\tilde f$:
  $$|f|_{s,t} = \tilde f(is\epsilon,t\epsilon), \quad
  \epsilon=(1,...,1).$$ 

  Now, the function 
  $$\varphi(z) := \tilde f(z\epsilon,e^{- (iz+s)}t\epsilon)$$
  is well defined on $\T_{s_1}$, for it is constant with respect to
  $\Re z$ and, due to the relations imposed on the norm indices, if
  $|\Im z|\leq s_1$ then $|e^{-(iz+s)}t| \leq e^{s_1-s} t = t_1$.
 
  The estimate
  $$\log \varphi(z) \leq \frac{s_1-\Im z}{s_1-s_0} \varphi(s_0 i)
  + \frac{\Im z-s_0}{s_1-s_0}\varphi (s_1 i)$$ trivially holds if $\Im
  z = s_0$ or $s_1$, for, as noted above for $j=1$, $e^{s_j-s} t
  = t_j$, $j=0,1$. But note that the left and right hand sides
  respectively are suharmonic and harmonic. Hence the estimate holds
  whenever $s_0 \leq \Im z \leq s_1$, whence the claim for $z=is$.
\end{proof}

Recall that we have let $\Tr_s^n := \T_s^n \times \D_s^n$, $s>0$, and,
for a function $f \in \mathcal{A}(\Tr_s^n)$, let $|f|_s = |f|_{s,s}$
denote its supremum norm on $\Tr_s^n$. As in the rest of the paper, we
now restrict the discussion to widths of analyticity $\leq 1$. 

\begin{corollary}\label{cor:interpolation}
  If $\sigma_1 = -\log \left( 1 - \frac{\sigma_0}{s} \right)$ and $f
  \in \mathcal{A}(\Tr_{s+\sigma_1}^n)$,
  $$|f|_{s}^2 \leq |f|_{s-\sigma_0} |f|_{s+\sigma_1}.$$
\end{corollary}

In~\ref{subsec:regularity}, we will use the equivalent fact that, if
$\tilde\sigma = s + \log \left( 1 + \frac{\sigma}{s} \right)$ and $f
\in \mathcal{A}(\Tr^n_{s+\tilde\sigma})$,
$$|f|_{s+\sigma}^2 \leq |f|_{s} |f|_{s+\tilde\sigma}.$$

\begin{proof}
  In proposition~\ref{prop:domains}, consider the following particular 
  case~: 
  \begin{itemize}
  \item $\rho = 1/2$. Hence
    $$s = \frac{s_0+s_1}{2} \quad \mbox{and} \quad t =
    \sqrt{t_0t_1}.$$ 
  \item $s=t$. Hence in particular $t_0 = s\, e^{s_0-s}$ and
    $t_1 = s \, e^{s_1-s}$. 
  \end{itemize}
  Then 
  $$|f|_s^2 = |f|_{s,s}^2 \leq |f|_{s_0,t_0} |f|_{s_1,t_1}.$$

  We want to determine $\max(s_0,t_0)$ and $\max(s_1,t_1)$. Let
  $\sigma_1 := s-s_0 = s_1-s$. Then $t_0 = s \, e^{-\sigma_1}$ and
  $t_1 = s \, e^{\sigma_1}$. The expression $s + \sigma - se^{\sigma}$
  has the sign of $\sigma$ (in the relevant region $0\leq s+\sigma
  \leq 1$, $0\leq s \leq 1$); by evaluating it at $\sigma=\pm
  \sigma_1$, we see that $s_0 \leq t_0$ and $s_1 \geq t_1$.

  Therefore, since the norm $|\cdot|_{s,t}$ is non-decreasing with
  respect to both $s$ and $t$,
  $$|f|_s^2 \leq |f|_{t_0,t_0} |f|_{s_1,s_1} = |f|_{t_0} |f|_{s_1}$$
  (thus giving up estimates uniform with respect to small values of
  $s$).  By further setting $\sigma_0 = s-t_0 = s\left(1 -
    e^{-\sigma_1} \right)$, we get the wanted estimate, and the
  asserted relation between $\sigma_0$ and $\sigma_1$ is readily
  verified.
\end{proof}




\section{Weaker arithmetic conditions of convergence}

In this section, we look more carefully to the arithmetic conditions
needed for the induction to converge, in the proof of the inverse
function theorem~\ref{thm:inversion}.

A function $\Delta : \N_* \rightarrow [1,+\infty[$ being given, define
the set $\mathrm{D}_\Delta$ as the subset of vectors $\alpha \in \R^n$
such that
$$|k \cdot \alpha| \geq \frac{(|k|+n-1)^{n-1}}{\Delta(|k|)} \quad (\forall k
\in \Z^n \setminus \{0\}).$$ (The function $\Delta$ is just some other
normalization of what is an approximation function in~\cite{Rus75} or
a zone function in \cite{Dumas:2004}.)  For $\mathrm{D}_\Delta$ to be
non empty, trivially we need $\lim_{+\infty} \Delta = +\infty$.

\begin{proposition}\label{prop:arithmetics}
  The conclusions of theorems~\ref{thm:herman}
  and~\ref{thm:kolmogorov} hold of there exist $c>0$ and $\delta \in
  ]0,1[$ such that
  $$\sum_{\ell \geq 1} \Delta(\ell) e^{-\ell / j^2} \leq \exp \left( c\,
    2^{\delta j} \right) \quad \mbox{as } j \rightarrow +\infty.$$
\end{proposition}

\begin{example}
  The Diophantine set $\mathrm{D}_{\gamma,\tau}$ corresponds to a
  polynomially growing function $\Delta$, and to a polynomially
  growing function $\sum_{\ell \geq 1} \Delta(\ell) e^{-\ell
    \,2^{-j}}$. A foriori, $\sum_{\ell \geq 1} \Delta(\ell)
  e^{-\ell/j^2}$ is at most polynomially growing.
\end{example}

\begin{proof}
  Call $L$ the discrete Laplace
  transform of $\Delta$:
  $$L(\sigma) = \sum_{\ell \geq 1} \Delta(\ell) e^{-\ell\sigma},$$
  and assume it is finite for all $\sigma >0$. Patterning the proof of
  lemma~\ref{lm:cohomological}, we get the following generalization.

  \begin{lemma}\label{lm:Cohomology}
    Let $g\in \mathcal{A}(\T^n_{s+\sigma})$ having $0$-average. There is
    a unique function $f\in \mathcal{A}(\T^n_s)$ of zero average such
    that $\mathcal{L}_\alpha f = g$. This function satisfies
    $$|f|_s \leq C \, L(\sigma) \, |g|_{s+\sigma}, \quad C = \frac{2^n
      e}{(n-1)!}.$$
  \end{lemma}

  (Again, see~\cite{Rus75} for improved estimates. But such an
  improvement is not the crux of our purpose here.)


  Taking up the proof of the inverse fuction theorem of appendix~A
  with our new estimates (see in particular
  equation~(\ref{eq:epsilon})), we see that the Newton algorithm
  converges provided
  $$\sum_{j\geq 0} 2^{-j} \log L(\sigma_j) < \infty,$$ 
  for some choice of the converging series $\sum \sigma_j$. Choosing
  $\sum\sigma_j = \sum j^{-2}$, we see that it is enough that $\log
  L(\sigma_j) \leq c \, 2^{\delta j}$ for some $c>0$ and $\delta
  \in ]0,1[$, whence the given criterion.
\end{proof}

\section{Comments}

\subsubsection*{Section~1.}

The proof of Kolmogorov's theorem presented here differs from others
chiefly for the following reasons:

-- The seeming detour through Herman's normal form reduces
Kolmogorov's theorem to a functionally well posed inversion problem
(compare with~\cite{Zeh76a,Zeh76b}). This powerful trick consists in
switching the frequency obstruction (obstruction to the conjugacy to
the initial dynamics) from one side of the conjugacy to the other. It
was extensively used in \cite{Mos67}. The remaining, finite
dimensional problem is then to show that the frequency offset $\beta
\in \R^n$ may vanish; in general, it is met using a non-degeneracy
hypothesis of one kind or another. Looking backward, this last step is
not the most difficult, but was probably not well understood before
M. Herman in the 80s (see \cite{Russmann:1990} and
\cite{Sevryuk:1999}). The functionnal setting chosen here adapts to
more degenerate cases, including lower dimensional tori, in a
straightforward manner (see~\cite{Fej04}; compare to Herman's prefered
proof for Lagrangian tori, as exposed in~\cite{Bos86}).

-- Classical perturbation series (or some modification of these) have
been shown to converge in some cases (\cite{Siegel:1942} for the
convergence of Schr{\"o}der series in the Siegel problem, see
\cite{Eliasson:1996} for Lindstedt series of Hamiltonians). Direct
methods for proving their convergence are involved because, as
J. Moser noticed in~\cite[p.~149]{Mos67}, these series do not converge
absolutely, and thus the proof of semi-convergence must take into
account compensations or the precise accumulation of small
denominators through a subtle combinatorial analysis. On the other
hand, the perturbation series yielded by the Newton algorithm are
absolutely convergent, provided that one adequatly chooses the width
of analytic spaces at each step of the induction. This was a major
discovery of Kolmogorov. In the first approximation, the series so
obtained can be thought of as obtained by grouping terms of the
classical perturbation series (from step $j$ to step $j+1$, the non
resonant terms of size $\epsilon^{2^j}, \cdots, \epsilon^{2^{j+1}-1}$
are eliminated). The magics is that compensations are taken into
account without noticing, and it would be interesting to understand
how classical and Newton series relate precisely, maybe with mould
calculus.

-- We encapsulate the Newton algorithm in an abstract inverse function
theorem {\`a} la Nash-Moser. The algorithm indeed converges without any
specific hypothesis on the internal structure of the variables. At the
expense of some optimality, ignoring this structure allows for simple
estimates (and control of the bounds) and for solving a whole class of
analogous problems with the same toolbox (lower dimensional tori,
codimension-one tori, Siegel problem, as well as some problems in
singularity theory).

-- The analytic (or Gevrey) category is simpler, in Nash-Moser theory,
than H{\"o}lder or Sobolev categories because the Newton algorithm can be
carried out without intercalating smoothing operators
(cf.~\cite{Ser72,Bos86}).

-- Incidentally, Hadamard interpolation inequalities are simple to
infer for analytic norms because, again, they do not depend on
regularizing operators, as it is shown in appendix~\ref{app:convexity}
(cf.~\cite[Theorem~A.5]{Hor76}).

-- The use of auxiliary norms ($|\cdot|_{G,s}$ in
lemmas~\ref{lm:cohomological} and~\ref{lm:seconde}, $|\cdot|_{x,s}$ in
appendix~\ref{app:submersion}) prevents from artificially loosing, due
to compositions, a fixed width of analyticity at each step of the
Newton algorithm --the domains of analyticity being deformed rather
than shrunk. As a pitfall, the argument of~\cite[Sections~5
and~6]{Jac72} to deduce an analytic function theorem in the smooth
category abstractly from the theorem in the analytic category, does
not apply directly here (see comment below).

\subsubsection*{Section~1. Theorem~\ref{thm:herman}} Herman's normal
form is the Hamiltonian analogue of the normal form of vector fields
on the torus in the neighborhood of Diophantine constant vector fields
(\cite{Arn61,Mos66}). The normal form for Hamiltonians implies the
normal form for vector fields on the torus~\cite[Th{\'e}or{\`e}me~40]{Fej04}
and is actually simpler to prove from the algebraic point of view.

\subsubsection*{Section~3. Lemma~\ref{lm:cohomological}} The estimate
is obtained by bounding the terms of Fourier series one by one. In a
more careful estimate, one should take into account the fact that if
$|k \cdot \alpha|$ is small, then $k' \cdot \alpha$ is not so small
for neighboring $k'$'s. This allows to find the optimal exponent of
$\sigma$, making it independant of the dimension; see
\cite{Mos66a} 
and \cite{Rus75}.

\subsubsection*{Appendix~\ref{app:submersion}.
  Theorem~\ref{thm:inversion}}

-- The two competing small parameters $\eta$ and $\sigma$ being fixed,
our choice of the sequence $(\sigma_n)$ maximizes $\epsilon$ for the
Newton algorithm. It does not modify the sequence $(x_k)$ but only the
information we retain from $(x_k)$.

-- In the expression of $\epsilon$, the square exponent of $C$ is
inherent in the quadratic convergence of Newton's algorithm. From this
follows the dependance, in KAM theory, of the size $\epsilon$ of the
allowed perturbation with respect to the small diophantine constant
$\gamma$: $\epsilon = O(\gamma^2)$.

-- The method of \cite{Jac72} (see~\cite{Mos66a} also in order to
deduce an inverse function theorem in the smooth category from its
analogue in the analytic category does not work directly, here. The
idea would be to use Jackson's theorem in approximation theory to
caracterize the H{\"o}lder spaces by their approximation properties in
terms of analytic functions and, then, to find a smooth preimage $x$
by $\phi$ of a smooth function $y$ as the limit of analytic preimages
$x_j$ of analytic approximations $y_j$ of $y$. However, in our
inversion function theorem we require the operator $\phi$ to be
defined only on balls $\sigma B_{s+\sigma}$ with shrinking radii when
$s+\sigma$ tends to $0$. This domain is too small in general to
include all the analytic approximations $y_j$ of a smooth $y$.  Such a
restriction is inherent in the presence of composition operators.
\cite{Jac72} did not have to deal with such operators for the problem
of isometric embeddings. Yet we \emph{could} generalize Jacobowitz's
proof at the expense of making additionnal hypotheses on the form of
our operator $\phi$, which would take into account the specificity of
directions $K$ and $G$, as well as of the real phase space and of its
complex extension.

\subsubsection*{Appendix~\ref{subsec:regularity}}

It is possible to prove that $\psi$ is $C^1$ without additional
asumptions, just by patterning~\cite[p.~626]{Ser72}). Yet the proof
simplifies and the estimates improve under the combined two additional
asumptions. In particular, the existence of a right inverse of
$\phi'(x)$ makes the inverse $\psi$ unique and thus allows to ignore
the way it was built.

\subsubsection*{Appendix~\ref{app:inversion}}

We include this elementary section for the sake of completeness,
although the quantitative estimates are needed only if one wants a
quantitative version of Kolmogorov's theorem, with an explicit value of
$\epsilon$.  A similar proposition (for germs at a point of maps in
$\C^n$) is proved in~\cite{Poe01} using a more sophisticated argument
from degree theory.

\subsubsection*{Appendix~\ref{app:convexity}}

In this paragraph, the obtained inequalities generalize the standard
Hada\-mard convexity inequalities. They are optimal and show that
analytic norms are not quite convex with respect to the width of the
complex extensions, due to the geometry of the phase space.
See~\cite[Chap.~8]{Nar71} for more general but less precise
inequalities.

\subsubsection*{Appendix~\ref{app:arithmetics}. Proposition~\ref{prop:arithmetics}} 
\label{app:arithmetics} 

There are reasons to believe that the so obtained arithmetic condition
is not optimal. Indeed, solving the exact cohomological equation at
each step is inefficient because the small denominators appearing with
intermediate-order harmonics deteriorate the estimates, whereas some
of these harmonics could have a smaller amplitude than the error terms
and thus would better not be taken care of. Even stronger, R{\"u}ssmann
and P{\"o}schel remarkably and recently noticed that at each step it is
worth neglecting part of the low-order harmonics themselves (to some
carefully chosen extent). Then the expense, a worse error term, turns
out to be cheaper than that the gain --namely, the right hand side of
the cohomological equation now has a smaller size over a larger
complex extension. This allows, with a slowly converging sequence of
approximations, to show the persistence of invariant tori under some
arithmetic condition which, in one dimension, is equivalent to the
Brjuno condition; see \cite{Poschel:2009}.

\bigskip \textit{Thank you to P. Bernard, A.  Chenciner, R.
  Krikorian, I. Kupka, D. Sauzin and J.-C. Yoccoz, for illuminating
  discussions, and to A. Albouy and A. Knauf for careful reading and
  correcting.}

\bibliographystyle{abbrvnat}
\bibliography{index} 

\end{document}